\documentclass[a4paper,11pt]{article}
\usepackage[T1]{fontenc}
\usepackage[utf8]{inputenc}
\usepackage[english]{babel}
\usepackage[margin=1.1in]{geometry}
\usepackage{microtype}
\usepackage{braket}
\usepackage{amsmath}
\usepackage{amssymb}
\usepackage{amsthm, aliascnt}
\usepackage{mathtools}
\usepackage{mathrsfs}
\usepackage{xcolor}
\usepackage[hyphens]{url}
\usepackage{hyperref}
\usepackage[hyphenbreaks]{breakurl}
\usepackage{enumitem}
\usepackage{tikz}
\usepackage[maxbibnames=99]{biblatex}
\usepackage{csquotes}
\usetikzlibrary{patterns}

\hypersetup{
                colorlinks=true,
                linkcolor={magenta},
                citecolor={cyan},
                breaklinks=true}

\theoremstyle{plain}
\newtheorem{teor}{Theorem}
\numberwithin{teor}{section}
\numberwithin{equation}{section}
\theoremstyle{definition}
\newaliascnt{defi}{teor}
\newtheorem{defi}[defi]{Definition}
\aliascntresetthe{defi}

\theoremstyle{plain}
\newaliascnt{lemma}{teor}%
\newtheorem{lemma}[lemma]{Lemma}
\aliascntresetthe{lemma}
\theoremstyle{plain}
\newaliascnt{prop}{teor}%
\newtheorem{prop}[prop]{Proposition}
\aliascntresetthe{prop}
\theoremstyle{plain}
\newaliascnt{cor}{teor}%

\aliascntresetthe{cor}
\theoremstyle{definition}
\newaliascnt{ex}{teor}%

\aliascntresetthe{ex}
\theoremstyle{definition}
\newaliascnt{oss}{teor}%
\newtheorem{oss}[oss]{Remark}
\aliascntresetthe{oss}
\DeclarePairedDelimiter{\abs}{\lvert}{\rvert}
\DeclarePairedDelimiter{\norma}{\lVert}{\rVert}
\DeclareMathOperator{\sbv}{SBV}
\DeclareMathOperator{\bv}{BV}
\newcommand{\marcomm}[1]{\marginpar{\begin{flushright}#1\end{flushright}}}%
\newcommand{\R}{\mathbb{R}}
\newcommand{\sbvv}{\sbv^{\frac{1}{2}}(\R^n)}
\newcommand{\N}{\mathbb{N}}
\newcommand{\Ln}{\mathcal{L}^n}
\newcommand{\Hn}{\mathcal{H}^{n-1}}
\newcommand{\eps}{\varepsilon}

\DeclareMathOperator{\divv}{div}
\DeclareMathOperator{\loc}{loc}

\makeatletter
\newcommand{\leqnomode}{\tagsleft@true\let\veqno\@@leqno}
\newcommand{\reqnomode}{\tagsleft@false\let\veqno\@@eqno}

\makeatother

\title{A free boundary problem in thermal insulation with a prescribed heat source}
\author{P. Acampora, E. Cristoforoni, C. Nitsch, C. Trombetti }
\date{}
\newcommand{\Addresses}{{%
 \bigskip 
 \footnotesize 
 
 \textsc{Dipartimento di Matematica e Applicazioni ``R. Caccioppoli'', Universit\`a degli studi di Napoli Federico II, Via Cintia, Complesso Universitario Monte S. Angelo, 80126 Napoli, Italy.}\par\nopagebreak 
 
 \medskip 
 
 \textit{E-mail address}, P.~Acampora: \texttt{paolo.acampora@unina.it} 
  
 \medskip 
 
 \textit{E-mail address}, C.~Nitsch: \texttt{c.nitsch@unina.it}

  \medskip 
 
 \textit{E-mail address}, C.~Trombetti: \texttt{cristina@unina.it} 
 
 \medskip 
 
\textsc{Mathematical and Physical Sciences for Advanced Materials and Technologies, Scuola Superiore Meridionale, Largo San Marcellino 10, 80126, Napoli, Italy.}\par\nopagebreak 
 
 \medskip 
 
 \textit{E-mail address}, E.~Cristoforoni: \texttt{emanuele.cristoforoni@unina.it} 
}}     

\begin{document}
\renewcommand*{\sectionautorefname}{Section}
\reversemarginpar
\maketitle

\begin{abstract}

We study the thermal insulation of a bounded body 
 $\Omega\subset\R^n$, under a prescribed
heat source $f>0$, via a bulk layer of insulating material. We consider a model of heat transfer between the insulated body and the environment determined by convection; this corresponds to Robin boundary conditions on the free boundary of the layer. We show that a minimal configuration exists and that it satisfies uniform density estimates.

 \textsc{MSC 2020:} 35R35, 35J25, 35A01.
 
\textsc{Keywords:} Robin, thermal insulation, free boundary, heat source
\end{abstract}

\section{Introduction}\label{introduction}
Let $\Omega\subseteq\R^n$ be an open bounded set with smooth boundary, let $f\in L^2(\Omega)$ be a positive function and let $\beta,\, C_0$ be positive constants. We consider the following energy functional
\begin{equation}\label{eq: fun 1}\mathcal{F}(A,v)=\int_A \abs{\nabla v}^2\,d\Ln +\beta\int_{\partial A } v^2\,d\Hn -2\int_\Omega fv\,d\Ln+C_0\Ln(A\setminus\Omega),\end{equation}
and the variational problem
\begin{equation}%
\label{problema}
\inf \Set{\mathcal{F}(A,v) | \begin{aligned}
&A\supseteq \Omega \text{ open, bounded and Lipschitz} \\
&v\in W^{1,2}(A), \: v\ge 0\,\text{in } A
\end{aligned}}.
\end{equation}

This problem is related to the following thermal insulation problem: for a given heat source $f$ distributed in a conductor $\Omega$, find the best possible configuration of insulating material surrounding $\Omega$. A similar problem has been studied in \cite{symbreak} and \cite{robinins} for a thin insulating layer, and in \cite{CK}, \cite{BucLuc12} and \cite{nahon} for a prescribed temperature in $\Omega$.

For a fixed open set $A$ with Lipschitz boundary, we have, via the direct methods of the calculus of variations, that there exists $u_A\in W^{1,2}(A)$ such that 
\[\mathcal{F}(A,u_A)\leq \mathcal{F}(A,v),\]
for all $v\in W^{1,2}(A)$, with $v\ge0$ in $A$. Furthermore $u_A$ solves the following stationary problem, with Robin boundary condition on $\partial A$%
. Precisely 
\[\begin{cases}-\Delta u_A= f & \text{in }\Omega,\\[3 pt]
\dfrac{\partial u_A^+}{\partial \nu\hphantom{\scriptstyle{+}}}=\dfrac{\partial u_A^-}{\partial \nu\hphantom{\scriptstyle{+}}} & \text{on }\partial\Omega, \\[6 pt]
\Delta u_A=0 & \text{in } A\setminus\Omega, \\[3 pt]
\dfrac{\partial u_A}{\partial \nu} +\beta u_A=0 & \text{on } \partial A,
\end{cases}\]
where $u_A^-$ and $u_A^+$ denote the traces of $u_A$ on $\partial\Omega$ in $\Omega$ and in $A\setminus\Omega$ respectively. That is 
\begin{equation}\label{eq: el}\int_A \nabla u_A\cdot \nabla\varphi\,d\Ln+\beta\int_{\partial A} u_A\varphi\,\Hn=\int_\Omega f\varphi\,d\Ln,\end{equation}
for all $\varphi\in W^{1,2}(A)$.  The Robin boundary condition represents the case when the heat transfer with the environment is conveyed by convection.

If for any couple $(A,v)$ with $A$ an open bounded set with Lipschitz boundary containing $\Omega$ and $v\in W^{1,2}(A)$, $v\geq0$ in $A$, we identify $v$ with $v\chi_A$, where $\chi_A$ is the characteristic function of $A$, and the set $A$ with the support of $v$, then the energy functional~\eqref{eq: fun 1} becomes 
\begin{equation}\label{eq: fun 2}\mathcal{F}(v)=\int_{\R^n} \abs{\nabla v}^2\,d\Ln+\beta\int_{J_v}\left( \overline{v}^2+\underline{v}^2\right)\,d\Hn-2\int_\Omega fv\,d\Ln+C_0\Ln(\set{v>0}\setminus\Omega),\end{equation}
and the minimization problem~\eqref{problema} becomes
\begin{equation}
\label{problemar}
\inf \Set{\mathcal{F}(v) | %
v\in \sbv^{\frac{1}{2}}(\R^n)\cap W^{1,2}(\Omega) 
},
\end{equation}
where $\overline{v}$ and $\underline{v}$ are respectively the approximate upper and lower limits of $v$, $J_v$ is the jump set and $\nabla v$ is the absolutely continuous part of the derivative of $v$. See \autoref{notations} for the definitions. 

We state the main results of this paper in the two following theorems.
\begin{teor}
\label{teor: mainth1}
Let $n\ge 2$, let $\Omega\subset\R^n$ be an open bounded set with  $C^{1,1}$ boundary, let $f\in L^2(\Omega)$, with $f>0$ almost everywhere in $\Omega$. Assume in addition that, if $n=2$,
\begin{equation}\label{eq:cond n=2}\norma{f}_{2,\Omega}^2<C_0\lambda_\beta(B)\mathcal{L}^2(\Omega),\end{equation}
where $B$ is a ball having the same measure of $\Omega$. Then problem~\eqref{problemar}
 admits a solution. Moreover,  if $p>n$ and $f\in L^p(\Omega)$, then there exists a positive constant $C=C(\Omega,f,p,\beta,C_0)$ such that if $u$ is a minimizer to problem~\eqref{problemar} then
 \[
 \norma{u}_\infty\le C.
 \]
\end{teor}
\begin{teor}
\label{teor: mainth2}
Let $n\ge 2$, let $\Omega\subset\R^n$ be an  open bounded set with $C^{1,1}$ boundary, let $p>n$ and let $f\in L^p(\Omega)$, with $f>0$ almost everywhere in $\Omega$. Assume in addition that, if $n=2$ condition~\eqref{eq:cond n=2} holds true. Then there exist positive constants  $\delta_0=\delta_0(\Omega,f,p,\beta,C_0)$, $c=c(\Omega,f,p,\beta,C_0)$, $C=C(\Omega,f,p,\beta,C_0)$ such that if $u$ is a minimizer to problem \eqref{problemar} then
\[
    u\ge\delta_0 \qquad \text{$\Ln$-a.e. in }\set{u>0},
\]
and the jump set $J_u$ satisfies the density estimates
\[cr^{n-1}\leq\Hn(J_u\cap B_r(x))\leq C r^{n-1},\]
with $x\in\overline{J_u}$, and $0<r<d(x,\partial\Omega)$. In particular, we have
\[
\Hn(\overline{J_u}\setminus J_u)=0.
\]
\end{teor}
We refer to \autoref{notations} for the definitions of $\lambda_\beta(B)$ in \eqref{eq:cond n=2}, and the distance $d(x,\partial\Omega)$ in \autoref{teor: mainth2}. 

\autoref{existence} is devoted to the proof \autoref{teor: mainth1}, while \autoref{estimates} is devoted to the proof \autoref{teor: mainth2}.

We notice that the assumptions on the function $f$ do not seem to be sharp. Indeed, it is well known that (see for instance \cite[Theorem 8.15]{trudinger}), in the more regular case, the assumption $f\in L^p(\Omega)$ with $p>n/2$ ensures the boundedness of solutions to equation~\eqref{eq: el}.%

\section{Notation and tools}\label{notations}
In this section we recall some definitions and proprieties of the spaces $\bv$, $\sbv$, and $\sbv^{\frac{1}{2}}$. We refer to \cite{bv}, \cite{robin-bg}, \cite{evans} 
for a deep study of the properties of these functions.

In the following, given an open set $\Omega\subseteq\R^n$ and $1\le p\le\infty$, we will denote the $L^p(\Omega)$ norm of a function $v\in L^p(\Omega)$ as $\norma{v}_{p,\Omega}$, in particular when $\Omega=\R^n$ we will simply write $\norma{v}_{p}=\norma{v}_{p,\R^n}$.

\begin{defi}[$\bv$]
Let $u\in L^1(\R^n)$. We say that $u$ is a function of \emph{bounded variation} in $\R^n$ and we write $u\in\bv(\R^n)$ if its distributional derivative is  a Radon measure, namely
\[
\int_{\Omega}u\,\frac{\partial\varphi}{\partial x_i}=\int_{\Omega}\varphi\, d D_i u\qquad \forall \varphi\in C^\infty_c(\R^n),
\]
with $Du$ a $\R^n$-valued measure in $\R^n$. We denote with $\abs{Du}$ the total variation of the measure $Du$. The space $\bv(\R^n)$ is a Banach space equipped with the norm
\[
\norma{u}_{\bv(\R^n)}=\norma{u}_{1}+\abs{Du}(\R^n).
\]

\end{defi}
\begin{defi}
Let $E\subseteq\R^n$ be a measurable set. We define the \emph{set of points  of density 1 for $E$} as 
\[
E^{(1)}=\Set{x\in\R^n | \lim_{r\to0^+}\dfrac{\Ln(B_r(x)\cap E)}{\Ln(B_r(x))}=1},
\]
and the \emph{set of points of density 0 for $E$} as 
\[
E^{(0)}=\Set{x\in\R^n | \lim_{r\to0^+}\dfrac{\Ln(B_r(x)\cap E)}{\Ln(B_r(x))}=0}.
\]
Moreover, we define the \emph{essential boundary} of $E$ as
\[
\partial^*E=\R^n \setminus(E^{(0)}\cup E^{(1)}).
\]
\end{defi}
\begin{defi}[Approximate upper and lower limits]
Let $u\colon\R^n\to\R$ be a measurable function. We define the \emph{approximate upper and lower limits} of $u$, respectively, as
\[\overline{u}(x)=\inf\Set{t\in\R|\limsup_{r\to0^+}\dfrac{\Ln(B_r(x)\cap\set{u>t})}{\Ln(B_r(x))}=0},\]
and
\[\underline{u}(x)=\sup\Set{t\in\R|\limsup_{r\to0^+}\dfrac{\Ln(B_r(x)\cap\set{u<t})}{\Ln(B_r(x))}=0}.\]
We define the \emph{jump set} of $u$ as 
\[J_u=\Set{x\in\R^n|\underline{u}(x)<\overline{u}(x)}.\] 
We denote by $K_u$ the closure of $J_u$. 
\end{defi} 
If $\overline{u}(x)=\underline{u}(x)=l$, we say that $l$ is the approximate limit of $u$ as $y$ tends to $x$, and we have that, for any $\eps>0$, 
\[\limsup_{r\to0^+}\dfrac{\Ln(B_r(x)\cap\set{\abs{u-l}\geq\eps)}}{\Ln(B_r(x))}=0.\]

If $u\in\bv(\R^n)$, the jump set $J_u$ is a $(n-1)$-rectifiable set, i.e. ${J_u\subseteq\bigcup_{i\in\mathbb{N}}M_i}$, up to a $\Hn$-negligible set, with $M_i$ a $C^1$-hypersurface in $\R^n$ for every $i$. We can then define $\Hn$-almost everywhere on $J_u$ a normal $\nu_u$ coinciding with the normal to the hypersurfaces $M_i$. Futhermore, the direction of $\nu_u(x)$ is chosen in such a way that the approximate upper and lower limits of $u$ coincide with the approximate limit of $u$ on the half-planes
\[H^+_{\nu_u}=\set{y\in\R^n|\nu_u(x)\cdot(y-x)\geq0}\]
and
\[H^-_{\nu_u}=\set{y\in\R^n|\nu_u(x)\cdot(y-x)\leq0}\]
respectively.
\begin{defi}
Let $E\subseteq\R^n$ be a measurable set and let $\Omega\subseteq\R^n$ be an open set. We define the \emph{relative perimeter} of $E$ inside $\Omega$ as
\[
P(E;\Omega)=\sup\Set{\int_E \divv\varphi\,d\Ln | \begin{aligned}
\varphi\in &\:C^1_c(\Omega,\R^n) \\ &\abs{\varphi}\le 1
\end{aligned}}.
\]
If $P(E;\R^n)<+\infty$ we say that $E$ is a \emph{set of finite perimeter}.
\end{defi}
\begin{teor}[Relative Isoperimetric Inequality]
\label{teor: relisop}
Let $\Omega$ be an open, bounded, connected set with Lipschitz boundary. Then there exists a positive constants $C=C(\Omega)$ such that %
\[
\min\Set{\Ln(\Omega\cap E),\Ln(\Omega\setminus E)}^\frac{n-1}{n}\le C P(E;\Omega),
\]
for every set $E$ of finite perimeter.
\end{teor}
See for instance \cite{mazya} for the proof of this theorem.
\begin{teor}
\label{teor: isopint}
Let $\Omega$ be an open, bounded, connected set with Lipschitz boundary. Then there exists a constant $C=C(\Omega)>0$ such that
\[
\Hn (\partial^*E\cap \partial\Omega )\le C \Hn(\partial^* E\cap \Omega) 
\]
for every set of finite perimeter $E\subset\Omega$ with $0<\Ln(E)\le \Ln(\Omega)/2$.
\end{teor}
We refer to \cite[Theorem 2.3]{cianchi2016poincare} for the proof of the previous theorem, observing that if $\Omega$ is a Lipschitz set, then it is an admissible set in the sense defined in \cite{cianchi2016poincare}(see \cite[Remark 5.10.2]{ziemer}).
\begin{teor}[Decomposition of $\bv$ functions]
Let $u\in\bv(\R^n)$. Then we have
\[
dDu=\nabla u\,d\Ln+\abs{\overline{u}-\underline{u}}\nu_u\,d\Hn\lfloor_{{\Huge J_u}}+ dD^c u,
\]
where $\nabla u$ is the density of $Du$ with respect to the Lebesgue measure, $\nu_u$ is the normal to the jump set $J_u$ and $D^c u$ is the \emph{Cantor part} of the measure $Du$. The measure $D^c u$ is singular with respect to the Lebesgue measure and concentrated out of $J_u$.
\end{teor}
\begin{defi}
Let $v\in\bv(\R^n)$, let $\Gamma\subseteq\R^n$ be a $\Hn$-rectificable set and let $\nu(x)$ be the generalized normal to $\Gamma$ defined for $\Hn$-a.e. $x\in\Gamma$. For $\Hn$-a.e. $x\in \Gamma$ we define the traces $\gamma_\Gamma^{\pm}(v)(x)$ of $v$ on $\Gamma$ by the following Lebesgue-type limit quotient relation
\[
\lim_{r\to 0}\frac{1}{r^n}\int_{B_r^{\pm}(x)}\abs{v(y)-\gamma_\Gamma^{\pm}(v)(x)}\,d\Ln(y)=0,
\]
where
\[
B_{r}^{+}(x)=\set{y\in B_r(x) | \nu(x)\cdot(y-x)>0},
\]
\[
B_{r}^{-}(x)=\set{y\in B_r(x) | \nu(x)\cdot(y-x)<0}.
\]\end{defi}
\begin{oss}
Notice that, by~\cite[Remark 3.79]{bv}, %
for $\Hn$-a.e. $x\in\Gamma$, $(\gamma_\Gamma^{+}(v)(x),\gamma_\Gamma^-(v)(x))$ coincides with either  $(\overline{v}(x),\underline{v}(x))$ or $(\underline{v}(x),\overline{v}(x))$, while, for $\Hn$-a.e. $x\in \Gamma\setminus J_v$, we have that $\gamma_\Gamma^+(v)(x)=\gamma_\Gamma^-(v)(x)$ and they coincide with the approximate limit of $v$ in $x$. In particular, if $\Gamma=J_v$, we have
\[
\gamma_{J_v}^+(v)(x)=\overline{v}(x) \qquad \gamma_{J_v}^-(v)(x)=\underline{v}(x)
\]
for $\Hn$-a.e. $x\in J_v$. 

\end{oss}
We now focus our attention on the $\bv$ functions whose Cantor parts vanish.
\begin{defi}[$\sbv$]
Let $u\in\bv(\R^n)$. We say that $u$ is a \emph{special function of bounded variation} and we write $u\in\sbv(\R^n)$ if $D^c u=0$. %
\end{defi}
For $\sbv$ functions we have the following.
\begin{teor}[Chain rule]\label{teor: chain}
Let $g\colon\R\to\R$ be a differentiable function. Then if $u\in\sbv(\R^n)$, we have
\[\nabla g(u)=g'(u)\nabla u.\]
Furthermore, if $g$ is increasing,
\[\overline{g(u)}=g(\overline{u}),\quad \underline{g(u)}=g(\underline{u})\]
 while, if $g$ is decreasing,
\[\overline{g(u)}=g(\underline{u}),\quad \underline{g(u)}=g(\overline{u}).\]
\end{teor}
We now give the definition of the following class of functions.
\begin{defi}[$\sbv^{1/2}$]
Let $u\in L^2(\R^n)$ be a non-negative function. We say that $u\in\sbvv$ if $u^2\in \sbv(\R^n)$. In addition, we define
\begin{gather*}
J_u:=J_{u^2} \qquad\quad \overline{u}:=\sqrt{\,\overline{u^2}\,}\qquad\quad \underline{u}:=\sqrt{\,\underline{u^2}\,} \\[5 pt]
\nabla u:=\frac{1}{2u}\nabla(u^2)\chi_{\set{u>0}}
\end{gather*}
\end{defi}
Notice that this definition extends the validity of the Chain Rule to the functions in $\sbvv$. We refer to \cite[Lemma 3.2]{robin-bg} for the coherence of this definition.
\begin{teor}[Compactness in $\sbv^{1/2}$]
\label{teor: sbv}
Let $u_k$ be a sequence in $\sbvv$ and let $C>0$ be such that for every $k\in \N$
\[
\int_{\R^n}\abs{\nabla u_k}^2\, d\Ln+\int_{J_{u_{k}}}\left(\overline u_k^2+\underline u_{k}^2\right)\,d\Hn +\int_{\R^n}u_k^2\, d\Ln < C
\]
Then there exists $u\in\sbvv$ and a subsequence $u_{k_j}$ such that
\begin{itemize}
    \item \emph{Compactness:}
    \[
    u_{k_j}\xrightarrow{L^2_{\loc}(\R^n)} u
    \]
    \item \emph{Lower semicontinuity:} for every open set $\Omega$ we have
    \[
    \int_\Omega \abs{\nabla u}^2\, d\Ln  \le \liminf_{j\to+\infty}\int_\Omega \abs{\nabla u_{k_j}}^2\,d\Ln
    \]
    and
    \[
    \int_{J_u\cap \Omega}\left( \overline u^2+\underline u^2\right)\, d\Hn  \le \liminf_{j\to+\infty}\int_{J_{u_{k_j}}\cap \Omega}\left( \overline u_{k_j}^2+\underline u_{k_j}^2\right)\, d\Hn
    \]
\end{itemize}
\end{teor}

\begin{defi}[Robin Eigenvalue]
\label{def: robineigen}
Let $\Omega\subseteq\R^n$ be an open bounded set with Lipschitz boundary, let $\beta>0$. We define $\lambda_{\beta}(\Omega)$ as
\begin{equation}\label{eq:Robin 2,q}\lambda_{\beta}(\Omega)=\inf\Set{ \dfrac{\displaystyle\int_\Omega \abs{\nabla v}^2\,d\Ln+\beta\int_{\partial \Omega} v^2\,d\Hn}{\displaystyle\int_\Omega v^2\,d\Ln}| v\in W^{1,2}(\Omega)\setminus\set{0}}.\end{equation}
\end{defi}

\begin{oss}
Standard tools of calculus of variation ensures that the infimum in~\eqref{eq:Robin 2,q} is achieved, see for instance. 
\end{oss}

\begin{lemma}
\label{lemma: stimaautovalori}
For every $0<r<R$, the following inequality holds
\[
\lambda_{\beta}(B_r)\le \left(\dfrac{ \Ln(B_R)}{\Ln(B_r)}\right)^{\frac{2}{n}}\lambda_{\beta}(B_R),
\]
where $B_R$ and $B_r$ are balls with radii $R$ and $r$ respectively.
\begin{proof}
Let $\varphi$ be a minimum  of~\eqref{eq:Robin 2,q} for $\Omega=B_R$ and with $\norma{\varphi}_{2,B_R}=1$. We define
\[
w(x)=\varphi\left(\frac{R}{r}x\right) \qquad \forall x\in B_r.
\]
Therefore, 
\[
\begin{split}
\lambda_{\beta}(B_r)%
&\le \dfrac{\displaystyle \int_{B_r}\abs{\nabla w(x)}^2\,d\Ln(x) + \int_{\partial B_r}w(x)^2\, d\Hn(x)}{\displaystyle \int_{B_r}w(x)^2\, d\Ln(x)}\\[7 pt]%
&= \dfrac{\displaystyle \left(\frac{r}{R}\right)^{n-2}\int_{B_R}\abs{\nabla \varphi(y)}^2\,d\Ln(y) + \left(\frac{r}{R}\right)^{n-1}\int_{\partial B_R}\varphi(y)^2\, d\Hn(y)}{\left(\dfrac{r}{R}\right)^{n}}.
\end{split}
\]
Since $r/R<1$, %
by minimality of $\varphi$, we get 
\[
\lambda_{\beta}(B_r)\le \dfrac{\displaystyle\left(\frac{r}{R}\right)^{n-2}}{\displaystyle\left(\dfrac{r}{R}\right)^{{n}\hphantom{\!2}}}\lambda_{\beta}(B_R)=\left(\dfrac{\Ln(B_r)}{\Ln(B_R)}\right)^{-\frac{2}{n}}\lambda_{\beta}(B_R).
\]
\end{proof}
\end{lemma}

Let $\beta,m>0$, and let us denote by
\[
\Lambda_{\beta,m}=\inf\Set{\dfrac{\displaystyle \int_{\R^n}\!\abs{\nabla v}^2\, d\Ln+\beta \int_{J_v}\!\!\left(\underline{v}^2+\overline{v}^2\right)\,d\Hn}{\displaystyle \int_{\R^n}v^2\, d\Ln} | \begin{aligned}
&v\in\sbvv\setminus{\{0\}} \\
&\Ln\left(\set{v>0}\right)\le m
\end{aligned}}.
\]
Here we state a theorem, referring to \cite[Theorem 5]{robin-bg} for the proof.
\begin{teor}%
\label{teor: faberkrahn}
Let $B\subseteq\R^n$ be a ball of volume $m$. Then
\[
\Lambda_{\beta,m}= \lambda_{\beta}(B).
\]
\end{teor}

We will denote by $d(x,\partial\Omega)$ the distance between $x\in\R^n$ and the boundary  $\partial\Omega$, and for every $\eps>0$ we define
\[
\Omega_\eps=\Set{x\in\Omega | d(x,\partial\Omega)>\varepsilon}.
\]
We will use the following result.
\begin{prop}
\label{teor: volumedistanza}
Let $\Omega$ be an open bounded set with $C^{1,1}$ boundary, then there exist a constant %
$C=C(\Omega)>0$ and a $\eps_0=\eps_0(\Omega)>0$ such that
\[
\Ln(\Omega\setminus\Omega_\eps)\le C\varepsilon \qquad \qquad \forall \eps<\eps_0.
\]
\end{prop}
\begin{proof}
It is well known (see for instance \cite[Theorem 17.5]{maggi}) that there exist a constant $C=C(\Omega)$ and $\eps_0=\eps_0(\Omega)>0$ such that 
\[P(\Omega_\eps)= P(\Omega)+C(\Omega)\,\eps+O(\eps^2),\]
for every $0<\eps<\eps_0$. Let $r(x)=d(x,\partial\Omega)$ be the distance from the boundary of $\Omega$. By coarea formula we have
\[\Ln(\Omega\setminus\Omega_\eps)=\int_{\set{0<r<\eps}} \,d\Ln=\int_0^\eps P(\Omega_t)\,dt\le C(\Omega)\eps.\]
\end{proof}
\section{Existence of minimizers}\label{existence}
In this section we prove \autoref{teor: mainth1}: in \autoref{teor: existence} we prove the existence of a minimizer to problem \eqref{problemar}; in \autoref{teor: linftybound} we prove the $L^\infty$ estimate for a minimizer. 

In this section we will assume that $\Omega\subseteq\R^n$ is an open bounded set with $C^{1,1}$ boundary, that $f\in L^2(\Omega)$ is a positive function and that $\beta, C_0$ are positive constants. We consider the energy functional $\mathcal{F}$ defined in \eqref{eq: fun 2}.%

\begin{lemma}
\label{lemma: stimel2} %
Let $n\ge2$ and assume that, if $n=2$, condition \eqref{eq:cond n=2} holds true. Then
there exist two positive constants $c=c(\Omega,f,\beta,C_0)$ and $C=C(\Omega,f,\beta,C_0)$ such that if $v\in\sbv^{\frac{1}{2}}(\R^n)\cap W^{1,2}(\Omega)$, with $\mathcal{F}(v)\le 0$ and $\Omega\subseteq\set{v>0}$, then
\begin{equation}\label{eq:stimasupporto}
\Ln(\Set{v>0})\le c,\end{equation}
\begin{equation}\label{eq:stimanormal2}
\norma{v}_2\le C.
\end{equation}
\end{lemma}
\begin{proof}
Let $B'$ be a ball with the same measure as $\Set{v>0}$. By \autoref{teor: faberkrahn}
\[
\begin{split}
0\ge\mathcal{F}(v)&\ge\,\lambda_\beta\left(B'\right)\int_{\R^n}v^2\,d\Ln-2\int_{\Omega}fv\, d\Ln\\ &\hphantom{\ge}+C_0\Ln\left(\set{v>0}\setminus\Omega\right).
\end{split}
\]
By \autoref{lemma: stimaautovalori} and Hölder inequality
\begin{equation}
\label{eq: quadratic}
\begin{split}
0\ge& \lambda_\beta(B)\left(\frac{\Ln(\Omega)}{\Ln(\Set{v>0})}\right)^{\frac{2}{n}}\norma{v}_2^2 -2\norma{f}_{2,\Omega}\norma{v}_2\\[3 pt] &+C_0\Ln\left(\set{v>0}\setminus\Omega\right)
\end{split}
\end{equation}
where $B$ is a ball with the same measure as $\Omega$. Obviously~\eqref{eq: quadratic} implies that
\[
\norma{f}_{2,\Omega}^2-\lambda_\beta(B)\left(\frac{\Ln(\Omega)}{\Ln(\Set{v>0})}\right)^{\frac{2}{n}}C_0\Ln\left(\set{v>0}\setminus\Omega\right)\ge 0.
\]
Let $M=\Ln(\set{v>0})$, and notice that, since $\Omega\subseteq\set{v>0}$,
\[
\Ln\left(\set{v>0}\setminus\Omega\right)= M-\Ln(\Omega),
\]
therefore
\[
\norma{f}_{2,\Omega}^2\ge C_0 \lambda_\beta(B)\left(\Ln(\Omega)\right)^{\frac{2}{n}}\left(M^{1-\frac{2}{n}}-M^{-\frac{2}{n}}\Ln(\Omega)\right).
\]
This implies (taking into account~\eqref{eq:cond n=2} if $n=2$) that there exists $c=c(\Omega,f,\beta,C_0)>0$ such that
\[
\Ln(\set{v>0})<c.
\]
Finally observe that by~\eqref{eq: quadratic} it follows
\begin{equation}
\label{eq: norma2}
\norma{v}_2\le C(M),
\end{equation}
where  \[\begin{split}C(M)&=\frac{M^{\frac{2}{n}}\left(\norma{f}_{2,\Omega}+\sqrt{\norma{f}_{2,\Omega}^2- C_0 \lambda_\beta(B)\left(\dfrac{\Ln(\Omega)}{M}\right)^{\frac{2}{n}}\left(M-\Ln(\Omega)\right)}\right)}{\vphantom{\big(}\lambda_\beta(B)\Ln(\Omega)}\\[5 pt]
&\le \frac{2 c^\frac{2}{n} \norma{f}_{2,\Omega}}{\lambda_\beta(B)\Ln(\Omega)}\end{split}\] %
\end{proof}

\begin{oss} 
\label{oss: ominsupp}
Let $v\in\sbvv\cap W^{1,2}(\Omega)$, it is always possible to choose a function $v_0$ such that $v_0=v$ in $\R^n\setminus\Omega$, $\mathcal{F}(v_0)\leq\mathcal{F}(v)$, and $\Omega\subseteq\set{v_0>0}$. Indeed the function $v_0\in W^{1,2}(\Omega)$, weak solution to the following boundary value problem 
\begin{equation}\label{eq: remark dir}\begin{cases} -\Delta v_0= f &\text{in }\Omega,\\
v_0 = \gamma^-_{\partial\Omega}(v) &\text{on }\partial\Omega,\end{cases}
\end{equation}
satisfies
\[\int_\Omega\nabla v_0\cdot\nabla \varphi\,d\Ln=\int_\Omega f\varphi\,d\Ln\]
for every $\varphi\in W^{1,2}_0(\Omega)$ and $v_0=\gamma^-_{\partial\Omega}(v)$ on $\partial\Omega$ in the sense of the trace. Then, extending $v_0$ to be equal to $v$ outside of $\Omega$, we have that $\Omega\subset\Set{v_0>0}$ and $\mathcal{F}(v_0)\leq\mathcal{F}(v)$.%
\end{oss}

\begin{prop}[Existence]
\label{teor: existence}
Let $n\ge2$ and, if $n=2$, assume that condition~\eqref{eq:cond n=2} holds true. Then there exists a solution to problem~\eqref{problemar}.
\end{prop}
\begin{proof}
Let $\{u_k\}$ be a minimizing sequence for problem~\eqref{problemar}. Without loss of generality we may always assume that, for all $k\in\N$, $\mathcal{F}(u_k)\leq\mathcal{F}(0)=0$, and, by \autoref{oss: ominsupp}, $\Omega\subseteq\set{u_k>0}$. Therefore we have 
\[\begin{split}0\geq\mathcal{F}(u_k)&\geq \int_{\R^n} \abs{\nabla u_k}^2\,d\Ln+\beta\int_{J_{u_k}}\left( \overline{u_k}^2+\underline{u_k}^2\right)\,d\Hn-2\int_\Omega fv\,d\Ln\\
&\geq\int_{\R^n} \abs{\nabla u_k}^2\,d\Ln+\beta\int_{J_{u_k}}\left( \overline{u_k}^2+\underline{u_k}^2\right)\,d\Hn-2\norma{f}_{2,\Omega}\norma{u_k}_{2,\Omega}\,,
\end{split}\]
and by~\eqref{eq:stimanormal2},%
\begin{equation*}
\int_{\R^n} \abs{\nabla u_k}^2\,d\Ln+\beta\int_{J_{u_k}}\left( \overline{u_k}^2+\underline{u_k}^2\right)\,d\Hn\leq C\norma{f}_{2,\Omega}\,.\end{equation*}

Then we have that there exists a positive constant still denoted by $C$, independent on the sequence $\{u_k\}$, such that
\begin{equation}\label{eq: variazione u}
\int_{\R^n}\abs{\nabla u_k}^2\, d\Ln+\int_{J_{u_{k}}}\left(\overline u_k^2+\underline u_k^2\right)\,d\Hn +\int_{\R^n}u_k^2\, d\Ln < C.
\end{equation}
The compactness theorem in $\sbvv$ (\autoref{teor: sbv}), ensures that there exists a subsequence $\{u_{k_j}\}$  and a function $u\in\sbvv\cap W^{1,2}(\Omega)$, such that $u_{k_j}$ converges to $u$ strongly in $L^2_{\loc} (\R^n)$, weakly in $W^{1,2}(\Omega)$, almost everywhere in $\R^n$ and
 \[\begin{split}
    \int_{\R^n}\abs{\nabla u}^2\, d\Ln  \le& \liminf_{j\to+\infty}\int_{\R^n} \abs{\nabla u_{k_j}}^2\,d\Ln,\\[5 pt]
    \int_{J_u}\left( \overline u^2+\underline u^2\right)\, d\Hn  \le& \liminf_{j\to+\infty}\int_{J_{u_{k_j}}} \left(\overline u_{k_j}^2+\underline u_{k_j}^2\right)\, d\Hn,\\[5 pt]
\Ln(\set{u>0}\setminus\Omega)\le&\liminf_{j\to+\infty}\Ln(\set{u_{k_j}>0}\setminus\Omega).\end{split}\]
Finally we have 
\[\mathcal{F}(u)\leq\liminf_{j\to+\infty}\mathcal{F}(u_{k_j})=\inf\Set{\mathcal{F}(v)|v\in\sbvv\cap W^{1,2}(\Omega)},\]
Therefore $u$ is a minimizer to problem~\eqref{problemar}.\qedhere
\end{proof}

\begin{teor}[Euler-Lagrange equation]
\label{teor: euler-lagrange}
Let $u$ be a minimizer to problem \eqref{problemar}, and let $v\in \sbv^{1/2}(\R^n)$ such that $J_v\subseteq J_u$. Assume that there exists $t>0$ such that  $\set{v>0}\subseteq \set{u>t}$ $\Ln$-a.e., and that
\[
\int_{J_u\setminus J_v}v^2\, d\Hn<+\infty.
\]
Then
\begin{equation}
\label{eq: wel}
\int_{\R^n}\nabla u\cdot\nabla v\, d\Ln+\beta\int_{J_u}\left(\overline{u}\gamma^+(v)+\underline{u}\gamma^-(v)\right)\, d\Hn=\int_\Omega fv\,d\Ln,
\end{equation}
where $\gamma^{\pm}=\gamma_{J_u}^\pm$.
\end{teor}
\begin{proof}
Notice that since $v\in\sbvv$ with $J_v\subseteq J_u$ we have that $v\in\sbvv\cap W^{1,2}(\Omega)$. Assume $v\in \sbv^{1/2}(\R^n)\cap L^\infty(\R^n)$. If $s\in\R$, recalling that $\set{v>0}\subseteq \set{u>t}$ $\Ln$-a.e., 
\[
u(x)+sv(x)=u(x)\ge 0 \qquad \text{$\Ln$-a.e.}\:\forall x\in \set{u\le t},
\]
while, for $\abs{s}$ small enough,
\[
u(x)+s v(x)\ge t-\abs{s}\,\norma{v}_\infty>0 \qquad \forall x\in\set{u>t}. 
\]
 Therefore we still have
\[
u+sv\in\sbv^{\frac{1}{2}}(\R^n,\R^+).
\]
Moreover by minimality of $u$ we have, for every $\abs{s}\le s_0$
\[\begin{split}
\mathcal{F}(u)\le&\mathcal{F}(u+sv)\\[3 pt]
=&\int_{\R^n}\abs{\nabla u+s\nabla v}^2\,d\Ln+\\[3 pt]
&+\int_{J_{u+sv}}\left[\left(\gamma^+(u)+s\gamma^+(v)\right)^2+\left(\gamma^-(u)+s\gamma^-(v)\right)^2\right]\,d\Hn+ \\[3 pt]
&-2\int_{\R^n}f(u+sv)\,d\Ln+C_0\Ln(\set{u>0}).
\end{split}
\]
\marcomm{\textbf{Claim:}}The set
\[S:=\Set{s\in[-s_0,s_0] | \, \Hn(J_{u}\setminus J_{u+sv})\ne 0}\]
is at most countable. 

\vspace{5 pt}
\noindent Let us define
\begin{gather*}
D_0=\Set{x\in J_u | \gamma^+(u)(x)\ne \gamma^-(u)(x)},\\[5 pt]
D_s=\Set{x\in J_u | \gamma^+(u+sv)(x)\ne\gamma^-(u+sv)(x)},
\end{gather*}
and notice that
\[
\Hn(J_u\setminus D_0)=0, \qquad \Hn(J_{u+sv}\setminus D_s)=0.
\]
Then we have to prove that
\[\Set{s\in[-s_0,s_0] | \, \Hn(D_0\setminus D_s)\ne 0}\]
is at most countable. Observe that if $t\ne s$, 
\[
(D\setminus D_t)\cap(D\setminus D_s)=\emptyset.
\]
Indeed if $x\in D\setminus D_s$
\begin{gather*}
\gamma^+(u)(x)\ne \gamma^-(u)(x),\\[5 pt]
\gamma^+(u)+s\gamma^+(v)(x)=\gamma^-(u)+s\gamma^-(v)(x),
\end{gather*}
then
\[
\gamma^+(v)(x)\ne \gamma^-(v)(x),
\]
and so
\[
s=\frac{\gamma^-(u)(x)- \gamma^+(u)(x)}{\gamma^+(v)(x)- \gamma^-(v)(x)}.
\]
If $\mathcal{H}^0$ denotes the counting measure in $\R$, we can write
\[
\int_{-s_0}^{s_0}\Hn(D_0\setminus D_s)\,d\mathcal{H}^0= \Hn\Bigg(\bigcup_{\scalebox{0.6}{$(-s_0,s_0)$}}\!D_0\setminus D_s\Bigg)\le\Hn(J_u)<+\infty,
\]
then the claim is proved. 

We are now able to differentiate in $s=0$ the function $\mathcal{F}(u+sv)$, and observing that $0\notin S$ is a minimum for $\mathcal{F}(u+sv)$, we get
\[
\frac{1}{2}\delta\mathcal{F}(u,v)=\int_{\R^n}\nabla u\cdot\nabla v\, d\Ln+\beta\int_{J_u}\left[\overline{u}\gamma^+(v)+\underline{u}\gamma^-(v)\right]\, d\Hn-\int_\Omega fv\,d\Ln=0.
\]
If $v\notin L^\infty(\R^n)$, we consider $v_h=\min\set{v,h}$. Then
\[
\delta\mathcal{F}(u,v_h)=0 \qquad \forall h>0.
\]
Observe that, since $\gamma^\pm(v_h)=\min\set{\gamma^\pm(v),h}$,
\[
\gamma^\pm(v_h)\to\gamma^\pm(v) \qquad \Hn\text{-a.e. in }J_u.
\]

Therefore, passing to limit for $h\to+\infty$, by dominated convergence on the term
\[
\int_{\R^n}\nabla u\cdot\nabla v_h\, d\Ln,
\]
and by monotone convergence on the terms
\[
\beta\int_{J_u}\left[\overline{u}\gamma^+(v_h)+\underline{u}\gamma^-(v_h)\right]\, d\Hn, \qquad\int_\Omega fv_h\,d\Ln,
\]
we get
\[
0=\lim_h\delta \mathcal{F}(u,v_h)=\delta\mathcal{F}(u,v).
\]
\end{proof}

We now want to use the Euler-Lagrange equation~\eqref{eq: wel} to prove that if $f$ belongs to $L^p(\Omega)$ with $p>n$, and if $u$ is a minimizer to problem~\eqref{problemar} then $u$ belongs to $L^{\infty}(\R^n)$. In order to prove this we need the following

\begin{lemma}
\label{lemma: poincare}
Let $m$ be a positive real number. There exists a positive constant $C=C(m,\beta,n)$ such that, for every function $v\in\sbvv$ with $\Ln(\set{v>0})\leq m$,
\[\left(\int_{\R^n} v^{2\cdot1^*}\,d\Ln\right)^\frac{1}{1^*}\leq C\left[\int_{\R^n} \abs{\nabla v}^2\,d\Ln+\beta\int_{J_v}\left(\overline{v}^2+\underline{v}^2\right)\,d\Hn\right],\]
where $1^*=\dfrac{n}{n-1}$ is the Sobolev conjugate of $1$.
\end{lemma}

\begin{proof}
Classical Embedding of $BV(\R^n)$ in $L^{1^*}(\R^n)$ ensures that
\[\begin{split}\left(\int_{\R^n} v^{2\cdot1^*}\,d\Ln\right)^\frac{1}{1^*}\leq& C(n) \abs*{D v^2}(\R^n)\\[5 pt]
=&C(n)\left[\int_{\R^n}2v\abs{\nabla v}\,d\Ln+\int_{J_v}\left(\overline{v}^2+\underline{v}^2\right)\,d\Hn\right].\end{split}\]
For every $\eps>0$, using Young's and Hölder's inequalities, we have
\[\begin{split}\left(\int_{\R^n} v^{2\cdot1^*}\,d\Ln\right)^\frac{1}{1^*}\leq& \frac{C(n)}{\eps}\int_{\R^n}v^2\,d\Ln+\\[3 pt] &+C(n)\left[\eps\int_{\R^n}\abs{\nabla v}^2\,d\Ln+\int_{J_v}\left(\overline{v}^2+\underline{v}^2\right)\,d\Hn\right]\\[5 pt]
\leq& \frac{C(n)\,m^{\frac{1}{n}}}{\eps}\left(\int_{\R^n} v^{2\cdot1^*}\,d\Ln\right)^\frac{1}{1^*}+\\[3 pt] &+C(n)\left[\eps\int_{\R^n}\abs{\nabla v}^2\,d\Ln+\int_{J_v}\left(\overline{v}^2+\underline{v}^2\right)\,d\Hn\right].
\end{split}\]
Setting $\eps=2C(n)m^{\frac{1}{n}}$, we can find two constants  $C(m,n),C(m,\beta,n)>0$ such that
\[\begin{split}\left(\int_{\R^n} v^{2\cdot1^*}\,d\Ln\right)^\frac{1}{1^*}&\leq C(m,n)\left[\int_{\R^n} \abs{\nabla v}^2\,d\Ln+\int_{J_v}\left(\overline{v}^2+\underline{v}^2\right)\,d\Hn\right]\\[5 pt]
&\leq C(m,\beta,n)\left[\int_{\R^n} \abs{\nabla v}^2\,d\Ln+\beta\int_{J_v}\left(\overline{v}^2+\underline{v}^2\right)\,d\Hn\right].\end{split}\]
\end{proof}

We refer to \cite{stampacchia} for the following lemma.

\begin{lemma}\label{lemma: stampacchia}
Let $g: [0,+\infty) \to [0,+\infty)$ be a decreasing function and assume that there exist $C,\alpha >0$ and $\theta>1$ constants such that for every $h>k\ge 0$,
\[
g(h)\le C(h-k)^{-\alpha}g(k)^\theta .
\]
Then there exists a constant $h_0>0$ such that
\[
g(h)=0 \qquad \forall h\ge h_0.
\]
In particular we have
\[
h_0=C^{\frac{1}{\alpha}}g(0)^{\frac{\theta -1}{\alpha}}2^{\theta(\theta-1)}.
\]
\end{lemma}

\begin{prop}[$L^\infty$ bound]
\label{teor: linftybound}
Let $n\ge2$ and assume that, if $n=2$, condition~\eqref{eq:cond n=2} holds true. Let $f\in L^p(\Omega)$, with $p>n$. Then there exists a constant $C=C(\Omega,f,p,\beta,C_0)>0$ such that if $u$ is a minimizer to problem \eqref{problemar}, then
\[
\norma{u}_{\infty}\le C.
\]
\end{prop}
\begin{proof}
Let $\gamma^{\pm}=\gamma_{J_u}^\pm$. For every $\varphi,\psi\in\sbvv$ satisfying $J_\varphi,J_\psi\subseteq J_u$, define
\[
a(\varphi, \psi)=\int_{\R^n}\nabla \varphi\cdot\nabla\psi\, d\Ln+\beta\int_{J_u}\left[\gamma^+(\varphi)\gamma^+(\psi)+\gamma^-(\varphi)\gamma^-(\psi)\right]\, d\Hn.
\]
For every $v$ satisfying the assumptions of \autoref{teor: euler-lagrange}, it holds that
\[
a(u,v)=\int_{\Omega}fv\,d\Ln.
\]
In particular, let us fix $k\in \R^+$ and define
\[
\varphi_k(x)=\begin{cases}
u(x)-k &\text{ if }u(x)\ge k,\\
0 &\text{ if }u(x)<k,
\end{cases}
\]
then
\[
\gamma^+(\varphi_k)(x)=\begin{cases} 
\overline u(x)-k &\text{ if }\overline u(x)\ge k,\\
0 &\text{ if }\overline u(x)<k,
\end{cases}
\]
and analogously for $\gamma^-(\varphi_k)$. Furthermore, let us define 
\[
\mu(k)=\Ln(\set{u>k}).
\]
We want to prove that $\mu(k)=0$ for sufficiently large $k$. From \autoref{teor: euler-lagrange}, we have
\begin{equation}
\label{eq: ela}
a(u,\varphi_k)=\int_{\Omega}f\varphi_k\,d\Ln,
\end{equation}
and we can observe that
\[\begin{split} a(u,\varphi_k)&=\int_{\set{u>k}}\abs{\nabla u}^2\,d\Ln+\beta\int_{J_u\cap\set{u>k}}\left[\overline{u}(\overline{u}-k)+\underline{u}(\underline{u}-k)\right]\,d\Hn \\
&\geq\int_{\set{u>k}}\abs{\nabla u}^2\,d\Ln+\beta\int_{J_u\cap\set{u>k}}\left[(\overline{u}-k)^2+(\underline{u}-k)^2\right]\,d\Hn\\[5 pt]&=a(\varphi_k,\varphi_k). \end{split}\]
Moreover, by minimality, $\mathcal{F}(u)\le \mathcal{F}(0)=0$ and by \autoref{oss: ominsupp}, $\Omega\subseteq\set{u>0}$. Therefore, \eqref{eq:stimasupporto} holds true  %
and we can apply \autoref{lemma: poincare}, having that there exists $C=C(\Omega,f,\beta,C_0)>0$ such that
\begin{equation}
\label{eq: bounda}
\int_{\Omega} f\varphi_k\,d\Ln =a(u,\varphi_k)\geq a(\varphi_k,\varphi_k)\geq C\norma{\varphi_k}_{2\cdot 1^*}^2.    
\end{equation}
On the other hand
\begin{equation}\label{eq: fbound}
\begin{split}
\int_\Omega f\varphi_k\,d\Ln=\int_{\Omega\cap\set{u>k}}f(u-k)\,d\Ln&\leq\left(\int_{\Omega\cap\set{u>k}}f^{\frac{2n}{n+1}}\right)^{\frac{n+1}{2n}}\norma{\varphi_k}_{2\cdot 1^*}\\[5 pt ]
&\leq\norma{f}_{p,\Omega}\,\norma{\varphi_k}_{2\cdot 1^*}\:\mu(k)^\frac{n+1}{2n\sigma'},
\end{split}
\end{equation}
where
\[
\sigma=\frac{p(n+1)}{2n}>1,
\]
since $p>n$. Joining~\eqref{eq: bounda} and~\eqref{eq: fbound}, we have
\begin{equation}\label{eq: prebound}\norma{\varphi_k}_{2\cdot 1^*}\leq C\norma{f}_{p,\Omega}\:\mu(k)^\frac{n+1}{2n\sigma'}.\end{equation}
Let $h>k$, then 
\[\begin{split}(h-k)^{2\cdot 1^*}\mu(h)=&\int_{\set{u>h}}(h-k)^{2\cdot 1^*}\,d\Ln\\
\leq&\int_{\set{u>h}}(u-k)^{2\cdot 1^*}\,d\Ln\\
\leq&\int_{\set{u>k}}(u-k)^{2\cdot 1^*}\,d\Ln=\norma{\varphi_k}_{2\cdot 1^*}^{2\cdot 1^*}.\end{split}\]
Using~\eqref{eq: prebound} and the previous inequality, we have
\[\mu(h)\leq C (h-k)^{-2\cdot 1^*}\mu(k)^\frac{n+1}{(n-1)\sigma'}.\]
Since $p>n$, then $\sigma'<(n+1)/(n-1)$. By \autoref{lemma: stampacchia}, we have that $\mu(h)=0$ for all $h\geq h_0$ with $h_0=h_0(\Omega, f,\beta,C_0)>0$, which implies 
\[\norma{u}_{\infty}\leq h_0.\]
\end{proof}
\begin{proof}[Proof of \autoref{teor: mainth1}]
The result is obtained by joining \autoref{teor: existence} and \autoref{teor: linftybound}.
\end{proof}
\section{Density estimates for the jump set}\label{estimates}
In this section we prove \autoref{teor: mainth2}: in \autoref{teor: lowerbound} we prove the lower bound for minimizers to problem \eqref{problemar}; in \autoref{teor: density1} and \autoref{teor: density2} we prove the density estimates for the jump set of a minimizer to problem \eqref{problemar}. 

In this section we will assume that $\Omega\subseteq\R^n$ is an open bounded set with $C^{1,1}$ boundary, that $f\in L^p(\Omega)$, with $p>n$, is a positive function, and that $\beta, C_0$ are positive constants. We consider the energy functional $\mathcal{F}$ defined in \eqref{eq: fun 2}.

 In order to show that if $u$ is a minimizer to problem~\eqref{problemar} then $u$ is bounded away from 0, we will first prove that there exists a positive constant $\delta$ such that $u>\delta$ almost everywhere in $\Omega$, and then we will show that this implies the existence of a positive constant $\delta_0$ such that $u>\delta_0$ almost everywhere in the set $\set{u>0}$. In the following we will denote by $U_t:=\Set{u<t}\cap\Omega$.

\begin{oss}
Let $u$ be a minimizer to~\eqref{problemar}, by \autoref{oss: ominsupp}, $u$ is a solution to 
\[\begin{cases}
-\Delta u=f & \text{in }\Omega,\\
u\ge0 & \text{on }\partial\Omega
\end{cases}\] 
Let $u_0\in W^{1,2}_0(\Omega)$ be the solution to the following boundary value problem
\begin{equation}\begin{cases}
\label{eq: elomega}
-\Delta u_0=f & \text{in }\Omega,\\
u_0=0 & \text{on }\partial\Omega.
\end{cases}\end{equation}
Then, by maximum principle,
\[
u\geq u_0\quad\text{in $\Omega\subseteq\set{u>0}$}\quad\text{and}\quad
\set{u<t}\cap\Omega=U_t\subseteq\set{u_0<t}\cap\Omega.
\]
\end{oss}

\begin{lemma}
\label{lemma: stimasottolivelli}
There exist two positive constants $t_0=t_0(\Omega,f)$ and $C=C(\Omega, f)$ such that if $u$ is a minimizer to~\eqref{problemar} then for every $t\in[0,t_0]$ it results
\begin{equation}
\label{eq: stimasottolivelli}
\Ln(U_t)\le C\, t.
\end{equation}
\begin{proof}
Let $u_0$ be the solution to \eqref{eq: elomega}, fix $\varepsilon>0$ such that the set
\[
\Omega_\eps=\Set{x\in\Omega | d(x,\partial\Omega)>\varepsilon}
\]
is not empty. Since $u_0$ is superharmonic and non-negative in $\Omega$, by maximum principle we have that \[\alpha=\inf_{\Omega_\eps}u_0>0.\]
then $u_0$ solves
\[
\begin{cases}
-\Delta u_0=f &\text{in } \Omega,\\
u_0 =  0 &\text{on } \partial\Omega, \\
u_0 \ge \alpha &\text{on } \partial\Omega_\eps.
\end{cases}
\]
Therefore, if we consider the solution $v$ to the following boundary value problem
\[
\begin{cases}
-\Delta v=0 &\text{in } \Omega\setminus\bar\Omega_\eps,\\
v =  0 &\text{on } \partial\Omega, \\
v = \alpha &\text{in } \bar\Omega_\eps,
\end{cases}
\]
we have that $u\ge u_0\ge v$ almost everywhere in $\Omega$ and
\[
\Set{u<t}\cap\Omega=U_t \subseteq\Set{u_0<t}\cap\Omega\subseteq\Set{v<t}\cap\Omega.
\]
Hopf Lemma implies that
there exists a constant $\tau=\tau(\Omega)>0$ such that
\[
\frac{\partial v}{\partial \nu}<-\tau\qquad \text{on }\partial\Omega. 
\]
Let $x\in\bar\Omega$, and let $x_0$ be a projection of $x$ onto the boundary $\partial\Omega$, then
\[
\abs{x-x_0}=d(x,\partial\Omega), \qquad \qquad \frac{x-x_0}{\abs{x-x_0}}=-\nu_\Omega(x_0),
\]
where $\nu_\Omega$ denotes the exterior normal to $\partial\Omega$. We can write
\begin{equation}\label{eq: lowboundboundary}\begin{split}
v(x)&=\underbrace{v(x_0)}_{=0}+\nabla v(x_0)\cdot(x-x_0)+o\left(\abs{x-x_0}\right)\\[7 pt]
&=-\frac{\partial v}{\partial\nu}(x_0)\abs{x-x_0}+o\left(\abs{x-x_0}\right) \\[5 pt]
&\ge \tau\abs{x-x_0}+o\left(\abs{x-x_0}\right)\\[5 pt]
&> \frac{\tau}{2}\abs{x-x_0}=\frac{\tau}{2} d(x,\partial\Omega)
\end{split}\end{equation}
for every $x$ such that $d(x,\partial\Omega)<\sigma_0$ for a suitable $\sigma_0=\sigma_0(\Omega,f)>0$. Notice that if $\bar{x}\in\overline{\Omega}$ and $\lim_{x\to \bar{x}}v(x)=0$
then necessarily $\bar{x}\in\partial\Omega$. Therefore, there exists a $t_0=t_0(\Omega,f)>0$ such that $v(x)<t_0$ implies $d(x,\partial\Omega)<\sigma_0$. %
Consequently, if $t<t_0$, we have that
\[
\set{v<t}\subseteq\set{d(x,\partial\Omega)<\sigma_0},
\]
and by \eqref{eq: lowboundboundary}, we get
\[
\Ln\left(U_t\right)\le \Ln\left(\Set{v<t}\right)\le \Ln\left(\Set{x\in\Omega | d(x,\partial\Omega)\le \frac{2}{\tau }\, t}\right).
\]
Since $\Omega$ is $C^{1,1}$, by \autoref{teor: volumedistanza}, we conclude the proof.
\end{proof}
\end{lemma}
\begin{lemma}\label{lemma: diffineq}
Let $g: [0,t_1] \to [0,+\infty)$ be an increasing, absolutely continuous function such that
\begin{equation}
\label{eq: diffineq}
g(t)\le C t^\alpha \left(g'(t)\right)^\sigma \qquad \forall t\in[0,t_1],
\end{equation}
with $C>0$ and $\alpha>\sigma>1$. Then there exists $t_0>0$ such that
\[
g(t)=0 \qquad \forall t\le t_0.
\]
Precisely,
\[
t_0=\left(\frac{C(\alpha-\sigma)}{\sigma-1}g(t_1)^{\frac{\sigma-1}{\sigma}}+t_1^{\frac{\sigma-\alpha}{\sigma}}\right)^{\frac{\sigma}{\sigma-\alpha}}.
\]
\end{lemma}

\begin{proof}
Assume by contradiction that $g(t)>0$ for every $t>0$. Inequality \eqref{eq: diffineq} implies
\[
\frac{g'}{g^{\frac{1}{\sigma}}}\ge \frac{1}{C}t^{-\frac{\alpha}{\sigma}}.
\]
Integrating between $t_0$ and $t_1$, we have
\[
\frac{\sigma}{\sigma-1}\left(g(t_1)^{\frac{\sigma-1}{\sigma}}-g(t_0)^{\frac{\sigma-1}{\sigma}}\right)\ge \frac{1}{C}\frac{\sigma}{\sigma-\alpha}\left(t_1^{\frac{\sigma-\alpha}{\sigma}}-t_0^{\frac{\sigma-\alpha}{\sigma}}\right).
\]
Since $\alpha>\sigma>1$, we have
\[
0\le g(t_0)^{\frac{\sigma-1}{\sigma}}\le \frac{\sigma-1}{C\left(\alpha-\sigma\right)}\left(t_1^{\frac{\sigma-\alpha}{\sigma}}-t_0^{\frac{\sigma-\alpha}{\sigma}}\right)+g(t_1)^{\frac{\sigma-1}{\sigma}},
\]
which is a contradiction if
\[
t_0\le \left(\frac{C(\alpha-\sigma)}{\sigma-1}g(t_1)^{\frac{\sigma-1}{\sigma}}+t_1^{\frac{\sigma-\alpha}{\sigma}}\right)^{\frac{\sigma}{\sigma-\alpha}}.
\]
\end{proof}

\begin{oss}\label{oss:deltanou}
Let $g$ be as in \autoref{lemma: diffineq} and assume that $g(t_1)\le K$, then $g(t)=0$ for all $0<t<\tilde{t}$ where
\[\tilde{t}=\left(\frac{C(\alpha-\sigma)}{\sigma-1}K^{\frac{\sigma-1}{\sigma}}+t_1^{\frac{\sigma-\alpha}{\sigma}}\right)^{\frac{\sigma}{\sigma-\alpha}}.
\]
\end{oss}
We now have the tools to prove the lower bound inside $\Omega$.
\begin{prop}\label{teor: bound Omega}
There exists a positive constant $\delta=\delta(\Omega, f,p,\beta,C_0)>0$ such that if $u$ is a minimizer to problem \eqref{problemar} then
\[
u\ge\delta\]
almost everywhere in $\Omega$.
\end{prop}
\begin{proof}
Assume that $\Omega$ is connected and define the function
\[
u_t(x)=\begin{cases}
\max\set{u,t} &\text{in }\Omega, \\
u &\text{in }\R^n\setminus\Omega
\end{cases}
\]
Recalling that $U_t=\set{u<t}\cap\Omega$, we have
\[
J_{u_t}\setminus\partial^*U_t = J_u\setminus\partial^*U_t,
\]
and on this set $\underline{u_t}=\underline{u}$ and $\overline{u_t}=\overline{u}$.
Then we get by minimality of $u$, and using the fact that $J_{u_t}\cap\partial^* U_t\subseteq\partial\Omega$,%
\[
\begin{split}
0&\ge \mathcal{F}(u)-\mathcal{F}(u_t)\\[7 pt] &=\int_{U_t}\abs{\nabla u}^2\, d\Ln -2\int_{U_t}f(u-t)\,d\Ln+\beta\int_{\partial^*U_t\cap J_u}\left(\underline{u}^2+\overline{u}^2\right)\,d\Hn +\\[5 pt]
&\hphantom{=}-\beta\int_{%
J_{u_t}\cap\partial^* U_t\cap J_u}\left[t^2+(\gamma_{\partial\Omega}^+(u))^2\right]\,d\Hn%
-\beta\int_{\left(J_{u_t}\cap\partial^*U_t\right)\setminus J_u} \left(t^2+u^2\right)\,d\Hn  \\[7 pt]
&\ge \int_{U_t}\abs{\nabla u}^2\, d\Ln-2\beta t^2 \Hn\left(\partial^*U_t\cap\partial\Omega\right)%
\end{split}
\]
where we ignored all the non-negative terms except the integral of $\abs{\nabla u}^2$, and we used that $u\le t$ in $\partial^* U_t\setminus J_u$. By \autoref{lemma: stimasottolivelli}, we can choose $t$ small enough to have $\Ln(U_t)\le \Ln(\Omega)/2$, then applying the isoperimetric inequality in \autoref{teor: isopint} to the set $E=%
U_t$, we get
\begin{equation}
\label{eq: stimagrad}
\int_{U_t}\abs{\nabla u}^2\, d\Ln\le 2\beta \,C\, t^2 \, P(U_t;\Omega).
\end{equation}
Let us define
\[
p(t)=P(U_t;\Omega),
\]
and consider the absolutely continuous function
\[
g(t)=\int_{U_t}u\,\abs{\nabla u}\, d\Ln=\int_0^t sp(s)\, ds.
\]
By minimality of $u$ we can apply the a priori estimates \eqref{eq: variazione u} to prove the equiboundedness of $g$, i.e. there exists $K=K(\Omega,f,\beta,C_0)>0$ such that  $g(t)\le K$ for all $t>0$. Using the Hölder inequality and the estimate \eqref{eq: stimagrad} we have 
\[
g(t)\le\left(\int_{U_t}u^2\, d\Ln\right)^\frac{1}{2}\left(\int_{U_t}\abs{\nabla u}^2\, d\Ln\right)^\frac{1}{2}\le \sqrt{2\beta C}\, t\,\Ln(U_t)^\frac{1}{2}(t^2p(t))^\frac{1}{2}.
\]
Fix $1>\varepsilon>0$. Then we can write $\Ln(U_t)=\Ln(U_t)^{\varepsilon}\,\Ln(U_t)^{1-\varepsilon}$, and by \autoref{lemma: stimasottolivelli} there exists a constant $C=C(\Omega, f,\beta)>0$ such that
\[
g(t)\le C\, t^{2+\frac{1-\varepsilon}{2}}\,\Ln(U_t)^{\frac{\varepsilon}{2}}\, p(t)^\frac{1}{2}.
\]
By the relative isoperimetric inequality in \autoref{teor: relisop}, we can estimate  
\[
\Ln(U_t)^\frac{\eps}{2}\le C(\Omega, n) p(t)^{\frac{\varepsilon n}{2(n-1)}},
\]
and, noticing that $p(t)=g'(t)/t$, we get
\[
g(t)\le C t^\alpha (g'(t))^\sigma,
\]
where
\[
\alpha=2-\frac{\varepsilon}{2}\left(1+\frac{n}{n-1}\right), \qquad \qquad \sigma=\frac{1}{2}+\frac{\varepsilon}{2}\frac{n}{n-1}.
\]
In particular, if we choose
\[
\varepsilon\in\left(\frac{n-1}{n},\frac{3n-3}{3n-1}\right),
\]
we have that $\alpha>\sigma>1$, and then, using \autoref{lemma: diffineq} and \autoref{oss:deltanou}, there exists a $\delta=\delta(\Omega, f,p,\beta,C_0)>0$ such that $g(t)=0$ for every $t<\delta$. Then $\Ln(\set{u<t}\cap\Omega)=0$ for every $t<\delta$, hence
\[u\ge\delta\]
almost everywhere in $\Omega$.

When $\Omega$ is not connected, then
\[
\Omega=\Omega_1\cup\dots\cup\Omega_N,
\]
with $\Omega_i$ pairwise disjoint connected open sets. Using $u_t$ as the function $u$ truncated inside a single $\Omega_i$, we find constants $\delta_i>0$ such that
\[
u(x)\ge \delta_i %
\]
almost everywhere in $\Omega_i$. Therefore choosing $
\delta=\min\set{\delta_1,\dots,\delta_N}$ we have $
u(x)>\delta$ almost everywhere in $\Omega$.
\end{proof}

Finally, following the approach in \cite{CK}, we have
\begin{prop}[Lower Bound]
\label{teor: lowerbound}  There exists a positive constant $\delta_0=\delta_0(\Omega,f,p,\beta,C_0)$ such that if $u$ is a minimizer to problem~\eqref{problemar} then
\[u\ge\delta_0\]
almost everywhere in $\Set{u>0}$.
\end{prop}
\begin{proof}
Let $\delta$ be the constant in \autoref{teor: bound Omega}. For every $0<t\le\delta$ let us define the absolutely continuous function
\[h(t)=\int_{\set{u\leq t}\setminus J_u}u\abs{\nabla u}\,\Ln=\int_0^t s P(\set{u>s};\R^n\setminus J_u)\,ds.\]
By minimality of $u$ we can apply the a priori estimates \eqref{eq: variazione u} to prove the equiboundedness of $h$, i.e. there exists $K=K(\Omega,f,\beta,C_0)>0$ such that  $h(t)\le K$ for all $t>0$.
We will show that $h$ satisfies a differential inequality.
For any $0<t<\delta$, let us consider $u^t=u\chi_{\set{u>t}}$, where $\chi_{\set{u>t}}$ is the characteristic funtion of the set $\set{u>t}$, as a competitor for $u$. We observe that, by \autoref{teor: bound Omega}, $\Omega\subseteq\set{u>t}$, so we have that
\[\begin{split}
   0 \ge& \mathcal{F}(u)-\mathcal{F}(u^t)\\[5 pt]
    =&\int_{\set{u\le t}\setminus J_u} \abs{\nabla u}^2\,\Ln+\beta\int_{J_u\cap\set{u>t}^{(0)}}\left( \underline{u}^2+\overline{u}^2\right)\,d\Hn +\\[5 pt] &+\beta\int_{J_u\cap\partial^*\set{u>t}} \underline{u}^2\,d\Hn -\beta\int_{\partial^*\set{u>t}\setminus J_u}{u}^2\,d\Hn +\\[5 pt]
    &+C_0\Ln\left(\Set{0<u\le t}\right).\end{split}\]
Rearranging the terms,
\begin{equation}
\label{eq: gest0}
\begin{split}
&\int_{\set{u\le t}\setminus J_u} \abs{\nabla u}^2\,\Ln+\beta\int_{J_u\cap\set{u>t}^{(0)}} \left(\underline{u}^2+\overline{u}^2\right)\,d\Hn +\\[5 pt] &+\beta\int_{J_u\cap\partial^*\set{u>t}} \underline{u}^2\,d\Hn  +C_0\Ln\left(\Set{0<u\le t}\right)\\[ 5 pt]
&\le \beta t^2 P(\set{u>t};\R^n\setminus J_u)=\beta t h'(t).
\end{split}
\end{equation}
On the other hand using Hölder's inequality, we have
\begin{equation}
\label{eq: gest1}
h(t)\le \left(\int_{\set{u\le t}} \abs{\nabla u}^2\,\Ln\right)^{\frac{1}{2}}\Ln\left(\Set{0<u\le t}\right)^{\frac{1}{2n}}\left(\int_{\set{u\le t}} u^{2\cdot 1^*}\,\Ln\right)^{\frac{1}{2\cdot 1^*} }.
\end{equation}
Classical Embedding of $\bv$ in $L^{1^*}$ applied to $u^2\chi_{\set{u\leq t}}$, ensures
\[\left(\int_{\set{u\le t}} u^{2\cdot 1^*}\,\Ln\right)^{\frac{1}{1^*} }\le C(n) \abs*{D \big(u^2\chi_{\set{u\le t}}\big)}(\R^n),\]
and, using \eqref{eq: gest0},
\begin{equation}
    \label{eq: gest2}
\begin{split}\abs*{D \big(u^2\chi_{\set{u\le t}}\big)}(\R^n)
    =&2\int_{\set{u\le t}} u\abs{\nabla u}\,\Ln+\int_{J_u\cap\set{u>t}^{(0)}}\left( \underline{u}^2+\overline{u}^2\right)\,d\Hn +\\[5 pt]& + \int_{J_u\cap\partial^*\set{u>t}} \underline{u}^2\,d\Hn +\int_{\partial^*\set{u>t}\setminus J_u}{u}^2\,d\Hn\\[5 pt]
    \le& 2t\left(\Ln\left(\Set{0<u\le t}\right)\int_{\set{u\le t}\setminus J_u} \abs{\nabla u}^2\,\Ln\right)^{\frac{1}{2}}+3t h'(t) \\[5 pt] \le& \left(2\frac{\delta\beta}{\sqrt{C_0}}+3\right)th'(t).
\end{split}
\end{equation}
Therefore, joining \eqref{eq: gest1}, \eqref{eq: gest0}, and \eqref{eq: gest2}, we have 
\begin{equation}\label{eq: gdiff}h(t)\le C_3\left(t h'(t)\right)^{1+\frac{1}{2n}},\end{equation}
where 
\[C_3=\beta^\frac{1}{2}\left(\frac{\beta}{C_0}\right)^\frac{1}{2n}C(n)^\frac{1}{2}\left(2\frac{\delta\beta}{\sqrt{C_0}}+3\right)^\frac{1}{2}.\]
By~\eqref{eq: gdiff} we now want to show that there exists $\delta_0=\delta_0(\Omega,f,p,\beta,C_0)>0$ such that $h(t)=0$ for every $0\le t<\delta_0$. Indeed assume by contradiction that $h(t)>0$ for every $0<t\le \delta$. We have
\[\frac{h'(t)}{h(t)^{\frac{2n}{2n+1}}}\ge \frac{C_3^{-\frac{2n}{2n+1}}}{t}.\]
Integrating from $t_0>0$ to $\delta$, we get
\[\left(h(\delta)^{\frac{1}{2n+1}}-h(t_0)^{\frac{1}{2n+1}}\right)\ge C_4 \log\left(\frac{\delta}{t_0}\right),\]
where \[C_4=\dfrac{C_3^{-\frac{2n}{2n+1}}}{2n+1}.\]
Then
\[h(t_0)^{\frac{1}{2n+1}}\le h(\delta)^{\frac{1}{2n+1}} + C_4\log\left(\frac{t_0}{\delta}\right).\]
Finally, for any 
\[
0<t_0\le \tilde{\delta}=\delta\exp\left(-h(\delta)^{\frac{1}{2n+1}}/C_4\right),
\]
we have $h(t_0)<0$, which is a contradiction. Then, setting $\delta_0=\delta\exp\left(-K^{\frac{1}{2n+1}}/C_4\right)\le\tilde{\delta}$, we conclude that $h(t)=0$ for any $0<t<\delta_0$, from which we have 
\[u\ge \delta_0\]
almost everywhere in $\set{u>0}$.
\end{proof}

\begin{oss}
From \autoref{teor: lowerbound}, if $u$ is a minimizer to problem~\eqref{problemar},  we have that \begin{equation}\label{eq: incljump}\partial^*\set{u>0}\subseteq J_u\subseteq K_u.\end{equation} Indeed, on $\partial^*\set{u>0}$ we have that, by definition,  $\underline{u}=0$ and that, since $u\ge\delta_0$ $\Ln$-a.e. in $\set{u>0}$, $\overline{u}\ge\delta_0$.
\end{oss}

\begin{prop}[Density Estimates]
\label{teor: density1}
There exist positive constants $C=C(\Omega,f,p,\beta, C_0)$, ${c=c(\Omega,f,p,\beta,C_0)}$  and $\delta_1=\delta_1(\Omega,f,p,\beta,C_0)$ such that if $u$ is a minimizer to problem \eqref{problemar} then for every $B_r(x)$ such that $B_r(x)\cap\Omega=\emptyset$, we have:%
\begin{enumerate}[label= (\alph*)]
\item For every $x\in\R^n\setminus\Omega$,
\begin{equation}
\label{eq: updensity}
\Hn(J_u\cap B_r(x))\le Cr^{n-1};
\end{equation}
\item For every $x\in K_u$,
\begin{equation}
\label{eq: lowNdensity}
\Ln(B_r(x)\cap\set{u>0})\ge cr^{n};
\end{equation}
\item The function $u$ has bounded support, namely
\[
\set{u>0}\subseteq B_{1/\delta_1}.
\]
\end{enumerate}
\end{prop}
\begin{proof}
This theorem is a consequence of \autoref{teor: lowerbound}, since we immediately have\marcomm{\phantom{-}\\ \it (a)}
\[
\int_{J_u\cap B_r(x)}\left(\underline{u}^2+\overline{u}^2\right)\,d\Hn\ge\delta_0^2\Hn(J_u\cap B_r(x)),
\]
and by minimality of $u$ we have
\[
\begin{split}
0&\ge\mathcal{F}(u)-\mathcal{F}(u\chi_{\R^n\setminus B_r(x)})\\ &\ge\int_{J_u\cap B_r(x)}\left(\underline{u}^2+\overline{u}^2\right)\,d\Hn-\int_{\partial B_r(x)\setminus J_u}{u}^2\,d\Hn\\
&\ge\int_{J_u\cap B_r(x)}\left(\underline{u}^2+\overline{u}^2\right)\,d\Hn-\int_{\partial B_r(x)\cap\set{u>0}^{(1)}}\left(\underline{u}^2+\overline{u}^2\right)\,d\Hn \\
&\ge\int_{J_u\cap B_r(x)}\left(\underline{u}^2+\overline{u}^2\right)\,d\Hn-2\norma{u}_\infty^2\Hn\left(\partial B_r(x)\cap\set{u>0}^{(1)}\right),
\end{split}
\]
where, in the second inequality, we have used~\eqref{eq: incljump}. Thus we have
\begin{equation}
\label{eq: updensity1}
\Hn\left(J_u\cap B_r(x)\right)\le\frac{2\norma{u}_\infty^2}{\delta_0^2}\Hn(\partial B_r(x)\cap\set{u>0}^{(1)})\le C r^{n-1},
\end{equation}
where $C=C(\Omega,f,p,\beta,C_0)>0$. 

We now\marcomm{\it (b)} want to use the estimate \eqref{eq: updensity} together with the relative isoperimetric inequality in order to get a differential inequality for the volume of $B_r(x)\cap\set{u>0}^{(1)}$. %
 Let $x\in K_u$, then for almost every $r$ we have%
\[
\begin{split}
0<V(r):=\Ln(B_r(x)\cap\set{u>0}^{(1)})&\le k\,P(B_r(x)\cap\set{u>0}^{(1)})^{\frac{n}{n-1}}\\
&\le k\, \Hn(\partial B_r(x)\cap\set{u>0}^{(1)})^{\frac{n}{n-1}},
\end{split}
\]
where $k=k(\Omega,f,p,\beta,C_0)>0$, and in the last inequality we used that \eqref{eq: incljump} and \eqref{eq: updensity1} imply
\[
\begin{split}
P\left(B_r(x)\cap\set{u>0}^{(1)}\right)&\le
\Hn\left(\partial B_r(x)\cap\set{u>0}^{(1)}\right)+\Hn\left(J_u\cap B_r(x)\right) \\
&\le \left(1+\frac{2\norma{u}_\infty^2}{\delta_0^2}\right) P(B_r(x);\set{u>0}^{(1)}).
\end{split}
\]
Then we have
\[
\frac{V'(r)}{V(r)^{\frac{n-1}{n}}}\ge \frac{1}{k},
\]
which implies
\[
\Ln(B_r(x)\cap\set{u>0}^{(1)})\ge c\,r^{n}.
\]
Finally, let\marcomm{\it (c)} $x\in K_u$ such that $d(x,\partial\Omega)\ge1/\delta_1$. From \eqref{eq: lowNdensity}, noticing that $\mathcal{F}(u)\le\mathcal{F}(0)=0$, we have that
\[
c\,\delta_1^{-n}\le\Ln(\set{u>0}\setminus\Omega)\le\frac{2\norma{u}_\infty}{C_0}\int_\Omega f\,d\Ln,
\]%
which is a contradiction if $\delta_1$ is sufficiently small. Then the thesis is given by \eqref{eq: incljump}.
\end{proof}
Finally, we have
\begin{prop}[Lower Density Estimate]
\label{teor: density2}
There exists a positive constant $c=c(\Omega,f,p,\beta,C_0)$ such that if $u$ is a minimizer to problem~\eqref{problemar} then
\begin{enumerate}
\item For any $x\in K_u$ and $B_r(x)\subseteq\R^n\setminus\Omega$,
\[
\Hn(J_u\cap B_r(x))\ge c r^{n-1};
\]
\item $J_u$ is essentially closed, namely
\[
\Hn(K_u\setminus J_u)=0;
\]
\end{enumerate}
\end{prop}
The proof of \autoref{teor: density2} relies on classical techniques used in \cite{Giorgi} to prove density estimates for the jump set of almost-quasi minimizers of the Mumford-Shah functional. We refer to \cite[Theorem 5.1]{CK} and \cite[Corollary 5.4]{CK} for the details of the proof.

\begin{proof}[Proof of \autoref{teor: mainth2}]
The result is obtained by joining \autoref{teor: lowerbound}, \autoref{teor: density1}, and \autoref{teor: density2}.
\end{proof}

\begin{oss}
Given the summability assumption on the function $f$, we have that minimizers to \eqref{problema} are quasi-minimizers of the functional $\mathcal{G}$, defined on $\sbvv\cap W^{1,2}(\Omega)$ as
\[\mathcal{G}(v)=\int_{\R^n}\abs{\nabla v}^2\,d\Ln+\beta\int_{J_v}(\overline{v}^2+\underline{v}^2)\,d\Hn,\]
that is, there exists $C(\Omega,f,p,\beta,C_0)>0$ and $\alpha(n,p)>n-1$ such that, if $B_r(x)$ is a ball of radius $r\le1$, and $v\in\sbvv\cap W^{1,2}(\Omega)$, with $\set{u\neq v}\subset B_r(x)$, then 
\[\mathcal{G}(u;B_r(x))\le \mathcal{G}(v;B_r(x))+C r^\alpha,\]
where
\[
\mathcal{G}(v;B_r(x)):=\int_{B_r(x)}\!\abs{\nabla v}^2\,d\Ln+\beta\!\int_{J_v\cap\overline{B_r(x)}}(\overline{v}^2+\underline{v}^2)\,d\Hn.
\]
Indeed, let $u$ be a minimizer to \eqref{problema}, let $B_r(x)$ be a ball of radius $r\le1$, and let $v\in\sbvv\cap W^{1,2}(\Omega)$, with $\set{u\neq v}\subset B_r(x)$; by the minimality of $u$ we have that
\[\begin{split}
  \mathcal{G}(u;B_r(x))-2\!\int_{\Omega\cap B_r(x)}fu\,d\Ln&\le\mathcal{G}(v;B_r(x))+C_0\Ln(\set{v>0}\cap B_r(x)\setminus\Omega)\\[3pt]
&\le\mathcal{G}(v;B_r(x))+Cr^n,
\end{split}\]
while
\[\int_{\Omega\cap B_r(x)}fu\,d\Ln\le\norma{f}_{p,\Omega}\norma{u}_\infty \Ln(B_r)^{1/p'}=C(\Omega,f,p,\beta,C_0)r^\alpha,\]
where \[n>\alpha=\frac{n}{p'}>n-1.\]
Finally, we have
\[\mathcal{G}(u;B_r(x)) \le \mathcal{G}(v;B_r(x))+C r^\alpha.\]
\end{oss}

\begin{oss}

Let $u$ be a minimizer to~\eqref{problemar} and let $A=\set{\overline{u}>0}\setminus K_u$, then the boundary of $A$ is equal to $K_u$: in first place, assume by contradiction that there exists an $x\in (\partial A)\setminus K_u$, then $u$ is superharmonic in a small ball centered in $x$ with radius $r$. Therefore, being
\[
\set{u>0}\cap B_r(x)\ne \emptyset,
\]
it is necessary that $u>0$ in the entire ball, and then $x\notin \partial A$, which is a contradiction. In other words, 
\[
\partial A\subseteq K_u
\]
By the same argument we also have that $A$ is open, and moreover $J_u\subseteq \partial A$, then
\[
K_u\subseteq \partial A.
\]
In particular, the pair $(A,u)$ is a minimizer for the functional
\[
\mathcal{F}(E,v)=\int_E\abs{\nabla v}^2\,d\Ln-2\int_{\Omega}fv\,d\Ln+\int_{\partial E}\left(\underline{v}^2+\overline{v}^2\right)\,d\Hn+C_0\Ln(E\setminus\Omega)
\]
over all pairs $(E,v)$ with $E$ open set of finite perimeter containing $\Omega$ and $v\in W^{1,2}(E)$. %
\end{oss}
\newpage

\printbibliography[heading=bibintoc]
\newpage
\Addresses
\end{document}